\documentclass[11pt,a4paper]{article}

\usepackage{amssymb,latexsym}
\usepackage{amsmath}               
\usepackage{amsfonts}       
\usepackage{amsthm}                
\usepackage{amssymb}
\usepackage{color}
\usepackage{verbatim}
\usepackage{comment}
\usepackage{graphicx}
\usepackage[margin=1.2in]{geometry}

\usepackage{lipsum}
\usepackage{authblk}
\usepackage{fancyhdr}
\renewenvironment{abstract}{%
\hfill\begin{minipage}{0.95\textwidth}
\rule{\textwidth}{1pt}}
{\par\noindent\rule{\textwidth}{1pt}\end{minipage}}
\makeatletter
\renewcommand\@maketitle{%
\hfill
\begin{minipage}{0.95\textwidth}
\vskip 2em
\let\footnote\thanks 
{\LARGE \@title \par }
\vskip 1.5em
{\large \@author \par}
\end{minipage}
\vskip 1em \par
}
\makeatother

\theoremstyle{plain}
\newtheorem{theorem}{Theorem}
\newtheorem{lemma}[theorem]{Lemma}
\newtheorem{corollary}[theorem]{Corollary}
\newtheorem{proposition}[theorem]{Proposition}

\theoremstyle{definition}

\def\C{\mathbb{C}}
\def\N{\mathbb{N}}
\def\Z{\mathbb{Z}}
\def\Q{\mathbb{Q}}

\DeclareMathOperator{\Mon}{Mon}
\DeclareMathOperator{\charp}{char}
\DeclareMathOperator{\Aut}{Aut}

\newcommand{\abs}[1]{\lvert#1\rvert}

\usepackage[colorlinks,pagebackref,pdftex, bookmarks=false]{hyperref}
\usepackage{indentfirst}

\begin{document}


\title{Diophantine equations in separated variables and lacunary polynomials}

\author{Dijana Kreso}
\affil{\small{Institute for Analysis und Number Theory,  Graz University of Technology, 
 Steyrergasse 30/II, 8010 Graz, Austria, and \\
Department of Mathematics, University of Salzburg, Hellbrunnerstrasse 34/I, 5020 Salzburg, Austria.\\
e-mail: \texttt{kreso@math.tugraz.at}}}

\maketitle


\begin{abstract}
{\bf Abstract:} 
We study Diophantine equations of type $f(x)=g(y)$, where $f$ and $g$ are lacunary polynomials. According to a well known finiteness criterion, for a number field $K$ and nonconstant $f, g\in K[x]$, the equation $f(x)=g(y)$ has infinitely many solutions  in $S$-integers $x, y$  only if $f$ and $g$ are representable as a functional composition of lower degree polynomials in a certain prescribed way. 
The behaviour of lacunary polynomials with respect to functional composition is a topic of independent interest, and has been studied by several authors. In this paper we utilize known results and develop some new results on the latter topic.
\end{abstract}

\vspace{0.5cm}

\section{Introduction}

The possible ways of writing a polynomial as a composition of lower degree polynomials were studied by several authors, starting with Ritt in the 1920's in his classical paper~\cite{R22}. The behaviour of lacunary polynomials with respect to functional composition has been studied by several authors, at least since the 1940's when Erd\H os and R\'enyi independently  investigated this topic. By a lacunary polynomial we mean a polynomial with a fixed number of nonconstant terms whose degrees of the terms and the coefficients may vary. By $a_{1}x^{n_{1}}+a_2x^{n_2}+\cdots+a_{\ell}x^{n_{\ell}}+a_{\ell+1}$ we denote a lacunary polynomial with $\ell$ nonconstant terms: Here we set a convention that in this notation $n_i$'s are positive integers such that $n_i>n_j$ if $ i>j$, and that $a_1a_2\cdots a_{\ell}\neq 0$, which we will use throughout the paper. In the last decade, various results are shown about the behaviour of lacunary polynomials (and rational functions) with respect to functional composition, see e.g.\@ \cite{FMZ15, FZ12,  Z07, Z08}. 

On the other hand, Diophantine equations of type $f(x)=g(y)$ have been of long-standing interest to number theorists. to classify. In 2000, Bilu and Tichy~\cite{BT00} classified polynomials $f, g$  for which the Diophantine equation $f(x)=g(y)$ has infinitely many solutions  in $S$-integers $x, y$,  by building on the work of  Ritt, Fried and Schinzel. It turns out that such $f$ and $g$ must be representable as a composition of lower degree polynomials in a certain prescribed way.

Here, in the light of the above results,  we are interested in Diophantine equations of type $f(x)=g(y)$, where $f$ and $g$ are lacunary. 
 Some results in this direction can be found in \cite{ G16+, K15, PPS11, S12}. 

To state our results, we introduce some notions. For a  number field $K$,  a finite set $S$ of places of $K$ that contains all Archimedean places and  the ring $\mathcal{O}_S$ of $S$-integers of $K$, we say that the equation $f(x)=g(y)$ has infinitely many solutions $x, y$ with a bounded $\mathcal{O}_S$-denominator if  there exists a nonzero $\delta\in \mathcal{O}_S$ such that there are infinitely many solutions $x, y\in K$ with $\delta x, \delta y\in \mathcal{O}_S$. 
Furthermore, we say that $f\in K[x]$ is  \emph{indecomposable} (over $K$) if $\deg f>1$ and $f$ cannot be represented as a composition of lower degree polynomials in $K[x]$. Otherwise, $f$ is said to be \emph{decomposable} (over $K$). Here is our first result.

\begin{theorem}\label{main}
Let $K$ be a number field, $S$ a finite set of places of $K$ that contains all Archimedean places and $\mathcal{O}_S$  the ring of $S$-integers of $K$. The equation
\begin{equation}\label{central2}
a_{1}x^{n_{1}}+a_2x^{n_2}+\cdots+a_{\ell}x^{n_{\ell}}+a_{\ell+1}=b_1y^{m_1}+b_2y^{m_2}+\cdots+b_{k}y^{m_k},
\end{equation}
where $\ell, k\geq 3$, $a_i, b_j\in K$, and
\begin{itemize} 
\item[i)] $\gcd(n_1, \ldots, n_{\ell})=1$, $\gcd(m_1, \ldots, m_{k})=1$,
\item[ii)] $b_1y^{m_1}+b_2y^{m_2}+\cdots+b_{k}y^{m_k}$ is indecomposable,
\item[iii)] $m_1\geq 2\ell(\ell-1)$, $m_1\neq k$ and $n_1\neq \ell$, and either $m_1\geq 2k+1$ or $n_1\geq 2\ell+1$, 
\end{itemize}
has infinitely many solutions $x, y\in K$ with a bounded $\mathcal{O}_S$-denominator if and only if 
\begin{equation}\label{triv}
a_{1}x^{n_{1}}+\cdots+a_{\ell}x^{n_{\ell}}+a_{\ell+1}=(b_1x^{m_1}+\cdots +b_kx^{m_k}) \circ \mu(x)
\end{equation}
for some linear $\mu\in K[x]$. 
\end{theorem}

Note that if in \eqref{triv} we have $\mu(0)=0$, then $k=\ell$, $n_i=m_i$, $a_{\ell+1}=0$ and $a_i=b_i\zeta$ for some $\zeta\in K\setminus\{0\}$ such that $\zeta^d=1$, where $d=\gcd(m_1, m_2, \ldots, m_k)$, for all $i=1, 2, \ldots, k$.  If $\mu(0)\neq 0$, then it can be shown that $n_1=m_1\leq k+\ell$, see Proposition~\ref{Gawron}. In Section~\ref{proofs}, we discuss how the assumptions in Theorem~\ref{main} arise, and in which way they can be relaxed at the cost of a more complicated formulation of the theorem. We also show a version of Theorem~\ref{main} when  $m_1$ is a composite number and  $iii)$ is relaxed to $m_1\geq 2\ell(\ell-1)$. 

Note that $ii)$ in Theorem~\ref{main} holds when $m_1$ is a prime (since if $f(y)=g(h(y))$, then $\deg f=\deg g\cdot \deg h$). Furthermore, $ii)$ in Theorem~\ref{main} holds when $ b_1m_1y^{m_1-1}+b_2m_2y^{m_2-1}+\cdots+b_{k}m_ky^{m_k-1}$ is irreducible over $K$ (since if $f(y)=g(h(y))$, then $f'(y)=g'(h(y))h'(y)$). 
Zannier~\cite{Z08} showed that if $K=\C$ and $ii)$ does not hold, then $(m_1, m_2, \ldots, m_k)\in M$, where  $M=M(b_1, b_2, \ldots, b_k)$ is  a finite union of subgroups of $\Z^k$. 
In \cite{DG06, DGT05}, it is shown that $ii)$ holds when $b_1, b_2, \ldots, b_k$ are nonzero integers, and either $m_2=m_1-1$ and $\gcd(m_1, b_2)=1$,  or $f$ is an odd polynomial, $m_2=m_1-2$ and $\gcd(m_1, b_2)=1$. (In the appendix we discuss an extension of the latter result to the case when the polynomial in $ii)$ has coefficients in any unique factorization domain). Furthermore, Fried and Schinzel~\cite{FS72} showed that  if $k=2$ and $\gcd(m_1, m_2)=1$, then $ii)$ holds. When $k=2$ we have the following result.

\begin{theorem}\label{main2}
Let $K$ be a number field, $S$ a finite set of places of $K$ that contains all Archimedean places and $\mathcal{O}_S$  the ring of $S$-integers of $K$.
The equation
\begin{equation}\label{central}
a_{1}x^{n_{1}}+\cdots+a_{\ell}x^{n_{\ell}}+a_{\ell+1}=b_1y^{m_1}+b_2y^{m_2},
\end{equation}
where $\ell\geq 3$, $a_i, b_j\in K$,  and
\begin{itemize} 
\item[i)] $\gcd(n_1, \ldots, n_{\ell})=1$, $\gcd(m_1,  m_2)=1$,
\item[ii)] $m_1\geq {\ell+2 \choose 2}+\ell-1$, $n_1\geq 3$, 
\end{itemize}
has  infinitely many solutions $x, y\in K$ with a bounded $\mathcal{O}_S$-denominator if and only if 
\begin{align}
\label{second}
\begin{split}
b_1x^{m_1}+b_2x^{m_2}&=e_1c(d_1x+d_0)x^{m_1-1} \\
a_{1}x^{n_{1}}+\cdots+a_{\ell}x^{n_{\ell}}+a_{\ell+1} &=e_1(c_1x+c_0)^{n_1},
\end{split}
\end{align}
for some $e_1, c, c_1, c_0, d_1, d_0\in K\setminus\{0\}$. 
\end{theorem}

Equation~\ref{central} was also studied in \cite{K15}, where a version of Theorem~\ref{main2} is shown under additional  assumptions. In this paper we utilize some new results, and in this way we improve the main result of \cite{K15}. 
To the proof of Theorem~\ref{main} of importance is a result of Zannier~\cite{Z07}, which states that for a field $K$ with $\textnormal{char}(K)=0$ and for $f\in K[x]$ with $\ell\geq 2$ nonconstant terms, which satisfies $f=g\circ h$ for some  $g, h\in K[x]$, where $h$ is not of type $ax^k+b$, we have $\deg g<2\ell(\ell-1)$. To the proof of Theorem~\ref{main2}, we show the following.
\begin{proposition}\label{nonlinear}
Let $K$ be a field with $\textnormal{char}(K)=0$. Assume that
\begin{equation}\label{linearsum}
a_1x^{n_1}+\cdots+a_{\ell}x^{n_{\ell}}+a_{\ell+1}=\left(b_1x^{m_1}+b_{2}x^{m_{2}}\right) \circ h(x),
\end{equation}
where $\ell\geq 3$, $a_i, b_j\in K$, $h\in K[x]$ and $\gcd(n_1, \ldots, n_{\ell})=1$.
Then 
\begin{equation}\label{bound}
m_1<{\ell+2 \choose 2}+\ell-1.
\end{equation}
If $a_{\ell+1}\neq 0$, then moreover
\begin{equation}\label{pbound}
m_1<{\ell+2 \choose 2}+2.
\end{equation}
\end{proposition}

Our proof of Proposition~\ref{nonlinear}, like Zannier's proof of the above mentioned result,  involves applying Brownawell and Masser's inequality~\cite{BM86}, which can be seen as a version of Schmidt's subspace theorem for function fields. 

Finally, we give a quick proof of the following theorem,  proved by P\'eter, Pint\'er and Schinzel~\cite{PPS11} in the case when $K=\Q$ and $\mathcal{O}_S=\Z$. 

\begin{theorem}\label{tri2}
Let $K$ be a number field, $S$ a finite set of places of $K$ that contains all Archimedean places and $\mathcal{O}_S$  the ring of $S$-integers of $K$. The equation
\begin{equation}\label{tritri}
a_1x^{n_1}+a_2x^{n_2}+a_3=b_1y^{m_1}+b_2y^{m_2},
\end{equation}
where $a_i, b_j\in K$, $\gcd(n_1, n_2)=1$, $\gcd(m_1, m_2)=1$, $m_1\geq 3$ and $n_1\geq 3$, 
has  infinitely many solutions with a bounded $\mathcal{O}_S$-denominator if and only if
\begin{equation}\label{lineartri2}
a_1x^{n_1}+a_2x^{n_2}+a_3=(b_1x^{m_1}+b_2x^{m_2})\circ \mu(x)
\end{equation}
 for some linear $\mu\in K[x]$. Furthermore, \eqref{lineartri2} with $\mu(0)\neq 0$ holds exactly when $n_1=m_1=3$, and either
\[
 n_2=m_2=2,\quad  a_1^2b_2^3+a_2^3b_1^2=0, \quad 27a_1^2a_3+4a_2^3=0,
\]
or 
\[
 n_2=2, \quad m_2=1, \quad 27a_1^4b_2^3+a_2^6b_1=0, \quad 3a_2^3a_3b_1+3a_1^2b_2^3+a_2^3b_2^2=0,
\]
and \eqref{lineartri2}  with $\mu(0)=0$ holds holds exactly when
\[
n_1=m_1, \quad n_2=m_2,\quad a_3=0,\quad a_1=b_1\zeta^{m_1},\ a_2=b_2\zeta^{m_2} \ \textnormal{for some}\ \zeta\in K\setminus\{0\}.
\]
\end{theorem}

To the proofs we deduce several results about decompositions of lacunary polynomials and we utilize the main result of Bilu and Tichy's paper~\cite{BT00}. We remark that the latter result relies on Siegel's classical theorem on integral points on curves, and is consequently ineffective.  Thus, our results are ineffective as well.

The paper is organized as follows. In Section~\ref{sec2} we recall the finiteness criterion from \cite{BT00} and some results on polynomial decomposition.
In Section~\ref{secD} we recall and prove several new results about decompositions of lacunary polynomials, and in particular we prove Proposition~\ref{nonlinear}.  In Section~\ref{proofs} we prove our main results using results from Section~\ref{sec2} and Section~\ref{secD}.

\section{Finiteness criterion}\label{sec2}

In this section we present the finiteness criterion of Bilu and Tichy~\cite{BT00}. 

 Let $K$ be a number field, $a, b\in K \setminus\{0\}$, $m, n\in \N$, $r\in \N\cup \{0\}$ and $p \in K[x]$ be a nonzero polynomial (which may be constant). Let further $D_{n} (x,a)$ be the $n$-th Dickson polynomial with parameter $a$ given by 
\begin{equation}\label{expdick}
D_n(x,a)=\sum_{j=0}^{\lfloor n/2 \rfloor} \frac{n}{n-j} {n-j \choose j} (-a)^{j} x^{n-2j}.
\end{equation}
We remark that $D_{n}(x,a)=2a^{n}T_{n}(x/(2\sqrt{a}))$ where $T_k(x)=\cos (k \arccos x)$ is the $k$-th Chebyshev polynomial of the first kind.
 For various properties of Dickson polynomials, see \cite[Sec.~3]{BT00}. 

To formulate the criterion, we need to define {\it standard} and {\it specific} pairs of polynomials.
{\it Standard} pairs of polynomials over $K$ are listed in the following table.
\vspace{0.2cm}
\begin{center}
\scalebox{0.8}{
 \begin{tabular}{|l|l|l|}
                \hline
                kind & standard pair (or switched) & parameter restrictions \\
                \hline
                first & $(x^m, a x^rp(x)^m)$ & $r<m, \gcd(r, m)=1,\  r+ \deg p > 0$\\
                second & $(x^2,\left(a x^2+b)p(x)^2\right)$ & - \\
                third & $\left(D_m(x, a^n), D_n(x, a^m)\right)$ & $\gcd(m, n)=1$\\
                fourth & $(a ^{\frac{-m}{2}}D_m(x, a), -b^{\frac{-n}{2}}D_n (x,b))$ & $\gcd(m, n)=2$\\
                fifth & $\left((ax^2 -1)^3, 3x^4-4x^3\right)$ & - \\
                \hline
        \end{tabular}}
\end{center}
\vspace{0.2cm}
We further call the pair
\[
\left (D_m \left (x, a^{n/d}\right), - D_n \left (x \cos (\pi/d), a^{m/d}\right)\right) \ \textnormal{(or switched)},
\]
with $d=\gcd(m, n)\geq 3$ and $\cos(2\pi/d)\in K$, a {\it specific pair} over $K$. One easily sees that if $b, \cos (2\alpha)\in K$, then $D_n(x\cos \alpha, b)\in K[x]$.

\begin{theorem}\label{T:BT}
Let $K$ be a number field, $S$ a finite set of places of $K$ that contains all Archimedean places, $\mathcal{O}_S$  the ring of $S$-integers of $K$, and $f, g\in K[x]$ nonconstant.
Then the following assertions are equivalent.
\begin{itemize}
\item[-] The equation $f(x)=g(y)$ has infinitely many solutions with a bounded $\mathcal{O}_S$-denominator;
\item[-] We have 
\begin{equation}\label{BTeq}
f(x)=\phi\left(f_{1}\left(\lambda(x)\right)\right)\quad \& \quad g(x)=\phi\left(g_{1}\left(\mu(x)\right)\right),
\end{equation}
where $\phi\in K[x]$, $\lambda, \mu\in K[x]$ are linear polynomials,
and $\left(f_{1},g_{1}\right)$ is a
standard or specific pair over $K$ such that the equation $f_1(x)=g_1(y)$
has infinitely many solutions with a bounded $\mathcal{O}_S$-denominator.
\end{itemize}
\end{theorem}

Recall that for a field $K$ a polynomial $f\in K[x]$ with $\deg f>1$ is called \emph{indecomposable} (over $K$)
if it cannot be written as the composition $f(x)=g(h(x))$ with $g,h\in K[x]$, $\deg g>1$ and $\deg h>1$. Otherwise, $f$ is said to be \emph{decomposable}. Any representation of $f$ as a functional composition of  polynomials of degree $>1$ is said to be a \emph{decomposition} of $f$. 
For a field $K$ with $\textnormal{char}(K)=0$ (and only those are of interest to us in this paper) and $f\in K[x]$ with $\deg f>1$, the Galois group of $f(x)-t$ over $K(t)$, where $t$ is transcendental over $K$, seen as a permutation group of the roots of this polynomial, is called the \emph{monodromy group} of $f$ (over $K$). The \emph{absolute monodromy group}  is the monodromy group of $f$ over an algebraic closure $\overline{K}$ of $K$. A lot of information about a polynomial is encoded into its monodromy group. In particular, $f(x)$ is indecomposable if and only if $\Mon(f)$ is a primitive permutation group. Furthermore, $(f(x)-f(y))/(x-y)\in K[x, y]$ is irreducible over $K$ if and only if $\Mon(f)$ is a doubly transitive permutation group. For the proofs of these facts see \cite{T95}. We now record the following property of Dickson polynomials that follows from \cite[Prop.\@ 1.7]{T95}.
\begin{lemma}\label{notDick}
Let $K$ be a field with $\charp(K)=0$ and let $f\in K[x]$ be such that the absolute monodromy group of $f$ is doubly transitive. If $\deg f\geq 4$, then there do not exist $e_i, c_i,a\in K$ such that $e_1c_1a\neq 0$ and $f(x)= e_1 D_n(c_1x+c_0, a)+e_0$. Furthermore, if $\deg f\geq 3$, there do not exist $e_i, c_i\in K$ such that $e_1c_1\neq 0$ and $f(x)= e_1 (c_1x+c_0)^k+e_0$.
\end{lemma}
From Lemma~\ref{notDick}  it follows that if the pair $(f, g)$ of polynomials with coefficients in a number field, with $\deg f\geq 3$ and $\deg g\geq 3$, is such that $f(x)=\phi\left(f_{1}\left(\lambda(x)\right)\right)$ and $g(x)=\phi\left(g_{1}\left(\mu(x)\right)\right)$ where $\phi, \lambda, \mu$ are all linear polynomials, and the absolute monodromy group of either $f(x)$ or $g(x)$ is doubly transitive, then $(f_1, g_1)$ cannot be of the third or of fourth kind. We will use this fact in the proofs of our main theorems.

\section{Lacunary polynomials}\label{secD}
In what follows, by $f^{(k)}$ we denote the $k$-th derivative of $f$.

\begin{lemma}[Haj\'{o}s's lemma]\label{Hajos}
Let $K$ be a field with $\textnormal{char}(K)=0$. If $f\in K[x]$ with $\deg f\geq 1$ has a root $\beta\neq 0$ of mutiplicity $m$, then $f$ has at least $m+1$ terms. 
\end{lemma}

A proof of Lemma~\ref{Hajos} can be found in e.g.\@ \cite[p.~187]{S00}. This is the main idea: Assume that $f$ has $\ell \leq m$ nonzero terms. Since the first $m$ derivatives of $f$ (i.e.\@ $f^{(0)}, \ldots, f^{(m-1)})$ vanish at $\beta$, we get a system of $\ell$ equations with $\ell$ unknowns (coefficients of $f$), for which one easily finds that its determinant is nonzero (as it reduces to Vandermonde type of determinant), so the system has a unique solution (trivial one), but the coefficients of $f$ are nonzero, a contradiction.

\begin{lemma}\label{choose}
Let $K$ be a field with $\textnormal{char}(K)=0$. 
Assume that
\[
a_1x^{n_1}+\cdots+a_{\ell}x^{n_{\ell}}+a_{\ell+1}=\left(b_1x^{m_1}+\cdots+b_{k}x^{m_{k}}+b_{k+1}\right) \circ \mu(x),
\]
where $\ell, k\geq 1$, $a_i, b_j\in K$, $\mu\in K[x]$, $\deg \mu=1$ and $\mu(0)\neq 0$.

 Then  for $i=2, 3, \ldots, \ell+1$, the $n_i$-th derivative of $b_1x^{m_1}+\cdots+b_{k}x^{m_{k}}+b_{k+1}$ 
has at least $n_{i-1}-n_{i}$ terms and the $(n_i+1)$-st derivative of $b_1x^{m_1}+\cdots+b_{k}x^{m_{k}}+b_{k+1}$ has a nonzero root of multiplicity $n_{i-1}-n_{i}-1$. Finally, if $n_i\geq m_i$ for $i=1, 2, \ldots, k$, then 
\[
n_1=m_1\leq k(k+1)/2.
\]
\end{lemma}
\begin{proof}
Let $f(x)=a_1x^{n_1}+\cdots+a_{\ell}x^{n_{\ell}}+a_{\ell+1}$ and $g(x)=b_1x^{m_1}+\cdots+b_{k}x^{m_{k}}+b_{k+1}$.
Let $n_{\ell+1}:=0$ and $m_{k+1}:=0$. Let further $\mu(x)=\alpha x+\beta$. By assumption $f(x)=g(\mu(x))$ and $\alpha, \beta\neq 0$. 

Note that $f^{(n_i)}(x)-f^{(n_i)}(0)=x^{n_{i-1}-n_{i}}h_i(x)$ for some $h_i\in K[x]$ for all $i=2, \ldots, \ell, \ell+1$. Since $f^{(n_i)}(x)=\alpha^{n_i}g^{(n_i)}(\alpha x+\beta)$, the last expression can be rewritten as 
\[
\alpha^{n_i}g^{(n_i)}(x)-f^{(n_i)}(0)=(x-\beta)^{n_{i-1}-n_{i}}\hat h_i(x)
\]
for some $\hat h_i\in K[x]$. So, $\beta\neq 0$ is a root of multiplicity $n_{i-1}-n_{i}$ of $\alpha^{n_i}g^{(n_i)}(x)-f^{(n_i)}(0)$. Thus, $\beta\neq 0$ is a root of multiplicity $n_{i-1}-n_{i}-1$ of $g^{(n_i+1)}(x)$  for all $i=2, \ldots, \ell, \ell+1$.

By Lemma~\ref{Hajos} it follows that $\alpha^{n_i}g^{(n_i)}(x)-f^{(n_i)}(0)$ has at least  $n_{i-1}-n_{i}+1$ terms, so   
$g^{(n_i)}(x)$ has at least  $n_{i-1}-n_{i}$ terms.

 If $n_i\geq m_i$ for $i=1, 2, \ldots, k$, then $\alpha^{n_i}g^{(n_i)}(x)-f^{(n_i)}(0)$ has at most $i$ terms. Then by Lemma~\ref{Hajos} it follows that $n_{i-1}-n_{i}+1\leq i$ for all $i=2, \ldots, k+1$. By taking sum, we get 
\[
n_1=\sum_{i=2}^{k+1} (n_{i-1}-n_i)\leq 1+2+\cdots+k={k+1 \choose 2}.
\]

\end{proof}
We remark that by Lemma~\ref{choose}, using the same notation, it follows that $n_1=m_1\leq \ell(k+1)$. Namely, $n_1=(n_1-n_2)+(n_2-n_3)\cdots+(n_{\ell}-n_{\ell+1})$, and if $n_1\geq \ell(k+1)+1$, then there exists $i\in \{2, \ldots, \ell+1\}$ such that $n_{i-1}-n_{i}\geq k+2$. However, $g^{(n_i)}(x)$ clearly has at most $k+1$ terms for any $i\in \{2, \ldots, \ell+1\}$, a contradiction.

Lemma~\ref{choose} is based on Zannier's Lemma~2 in \cite{Z07}. Zannier studied the case  when $f=g$ using similar arguments. The following result obtained by  Gawron~\cite{G16+} improves Zannier's lemma. Gawron's proof is based on a classical result of Gessel and Viennot about matrices with binomial coefficients.  In that paper, Gawron studied Equation~\ref{central2} when $k=3$ and $\ell\geq 4$, and when $k=\ell=3$.

\begin{lemma}\label{Gawron}
Let $K$ be a field with $\textnormal{char}(K)=0$ and let $f\in K[x]$ have $\ell\geq 1$ nonconstant terms and $g\in K[x]$ have $k\geq 1$ nonconstant terms. Assume that $f(x)=g(\mu(x))$ for some linear $\mu\in K[x]$ such that $\mu(0)\neq 0$. Then $\deg f=\deg g\leq k+\ell$. In particular, if $f=g$, then $\deg f\leq 2\ell$.
\end{lemma}

The following result is due to Zannier~\cite{Z07}. 

\begin{theorem}\label{Zannier}
Let $K$ be a field with $\textnormal{char}(K)=0$ and let $f\in K[x]$ have $\ell\geq 1$ nonconstant terms. Assume that $f=g\circ h$, where $g, h\in K[x]$ and where $h$ is not of type $ax^k+b$ for $a, b\in K$. If $\ell\geq 2$, then $\deg g<2\ell(\ell-1)$. If $\ell=1$, then $\deg g=1$.
\end{theorem}

Let $f\in K[x]$ with $\ell\geq 1$ nonconstant terms  be decomposable. Write $f(x)=g(h(x))$ with $g, h\in K[x]$, $\deg g \geq 2$, $\deg h\geq 2$, $h$ monic and $h(0)=0$. Theorem~\ref{Zannier} implies that if $\ell=1$, then $h(x)=x^k$, and if $\ell\geq 2$, then 
either $\deg g< 2\ell(\ell-1)$ or $h(x)=x^k$. 
Note that 
\[
a_1x^{n_1}+a_2x^{n_2}+\cdots+a_{\ell}x^{n_{\ell}}+a_{\ell+1}=f(x)=g(x)\circ x^k,
\]
exactly when $k\mid n_i$ for all $i=1, 2, \ldots, \ell$.

The main ingredients of Zannier's proof of Theorem~\ref{Zannier} are Lemma~\ref{Hajos} and the following result of Brownawell and Masser~\cite{BM86}, which can be seen as a version of Schmidt's subspace theorem for function fields. 

\begin{theorem}\label{vanish}
Let $K'/K(x, y)$ be a function field of one variable of genus $g$, and let $z_1, \ldots, z_s\in K'$ be not all constant and such that $1+z_1+\cdots+z_s=0$. Suppose also that no proper subsum of the left side vanishes. Then
\[
\max (\deg (z_i))\leq {s \choose 2}\left(\# S+2g-2\right), 
\] 
where $S$ is a set of points of $K'$ containing all zeros and poles of the $z_i$'s.
\end{theorem}

Further note that if $\ell=2$ in  Theorem~\ref{Zannier}, then if $f=g\circ h$, where $g, h\in K[x]$ and where $h$ is not of type $ax^k+b$, we have that $\deg g\leq 3$. 
Fried and Schinzel~\cite{FS72} have shown that in this case $\deg g=1$.
 In particular, if $\gcd(n_1, n_2)=1$, then  $a_1x^{n_1}+a_2x^{n_2}+a_3\in K[x]$  is indecomposable. Turnwald~\cite{T95} showed that the following stronger result holds.

\begin{proposition}\label{Montri}
Let $K$ be a field with $\textnormal{char}(K)=0$ and $f(x)=a_1x^{n_1}+a_2x^{n_2}+a_3\in K[x]$, with $\gcd(n_1, n_2)=1$. Then $\Mon(f)$ is symmetric.
\end{proposition}

We now prove Proposition~\ref{nonlinear}.

\begin{proof}[Proof of Proposition~\ref{nonlinear}]
Assume first that $\deg h\geq 2$. Let $z=h(x)$. Then
\begin{equation}\label{zerosum}
a_1x^{n_1}+\cdots+a_{\ell}x^{n_{\ell}}+a_{\ell+1}=b_1z^{m_1}+b_{2}z^{m_{2}}.
\end{equation}

We will make use of Theorem~\ref{vanish}.
Assume that there exists a proper vanishing subsum of \eqref{zerosum}. Choose a vanishing subsum which involves $a_1x^{n_1}$ and has no proper vanishing  subsum, and further write this vanishing sum as 
$p(x)=q(z)$. Clearly, $\deg p=n_1$, the number of terms of $p$ is $\leq \ell$, and by comparison of the degrees we have $q(z)=b_1z^{m_1}$. Thus, $p(x)=b_1x^{m_1}\circ h(x)$. By Lemma~\ref{Hajos}, it follows that either $h$ has no nonzero root, i.e.\@ $h(x)=cx^k$ for some $k\in \N$ and $c\in K\setminus\{0\}$, or $m_1\leq \ell-1$. In the latter case, we get what we sought and more. In the former case,  since $\gcd(n_1, \ldots, n_{\ell})=1$, it must be that $k=1$,  which contradicts the assumption $\deg h\geq 2$. 

Assume henceforth that there exists no proper vanishing subsum of \eqref{zerosum}. Note that $x, z\in K(x)$. From \eqref{zerosum} it follows that
\[
\frac{a_1x^{n_1}}{a_{\ell}x^{n_{\ell}}}+\cdots+1+\frac{a_{\ell+1}}{a_{\ell}x^{n_{\ell}}}
-\frac{b_1h(x)^{m_1}}{a_{\ell}x^{n_{\ell}}}-\frac{b_2h(x)^{m_2}}{a_{\ell}x^{n_{\ell}}}=0.
\]
Note that the total number of zeros and poles  of the terms in the above vanishing sum is at most $\deg h+1$. By Theorem~\ref{vanish}, it follows that
\[
n_1-n_{\ell}\leq {\ell+2 \choose 2}(\deg h+1+2\cdot 0-2)={\ell+2 \choose 2}(\deg h-1).
\]
Write $b_1x^{m_1}+b_{2}x^{m_{2}}-a_{\ell+1}=b_1\prod_{i=1}^r(x-\beta_i)^{e_i}$, with distinct $\beta_i$'s and positive integers $e_i$. Then
\begin{equation}\label{beta_i}
a_1x^{n_1}+\cdots+a_{\ell}x^{n_{\ell}}=b_1\prod_{i=1}^r(h(x)-\beta_i)^{e_i}.
\end{equation}

Since the factors in the product are coprime, it follows that $x^{n_{\ell}}$ divides $(h(x)-\beta_i)^{e_i}$ for some $i$, say $i_0$. Since by assumption $h(x)-\beta_{i_0}$ has at least one nonzero root (since $\deg h\geq 2$ and $\gcd(n_1, \ldots, n_{\ell})=1$), it follows that $n_{\ell}\leq (\deg h-1)\cdot e_{i_0}$. 

By Lemma~\ref{Hajos}, from $b_1x^{m_1}+b_{2}x^{m_{2}}-a_{\ell+1}=b_1\prod_{i=1}^r(x-\beta_i)^{e_i}$ we have that $e_i\leq 2$ for all $i$ if $a_{\ell+1}\neq 0$. If $a_{\ell+1}=0$, then $\beta_{i_0}=0$, $e_{i_0}=m_2$ and $e_i=1$ for all $i$ except for $i_0$. However, since $h(x)-\beta_{i_0}$ must have a nonzero root (again, since $\deg h\geq 2$ and $\gcd(n_1, \ldots, n_{\ell})=1$), it follows from \eqref{beta_i} that $e_{i_0}\leq \ell-1$. Thus, $n_{\ell}\leq (\ell-1)(\deg h-1)$.


Therefore,
\begin{align*}
n_1\leq {\ell+2 \choose 2}(\deg h-1)+n_{\ell}&\leq {\ell+2 \choose 2}(\deg h-1)+(\ell-1) (\deg h-1)\\
&=(\deg h-1)\left ({\ell+2 \choose 2}+\ell-1\right), 
\end{align*}
so in particular \eqref{bound} holds. Clearly, if $a_{\ell+1}\neq 0$, then by what we showed above, the summand $\ell-1$ in the sum above can be replaced by $2$, so \eqref{pbound} holds.

Let now $\deg h=1$. Clearly,  if $h(0)=0$, then $\ell=2$, a contradiction with the assumption.  Thus, $h(0)\neq 0$. Then the polynomial on the right hand side of \eqref{linearsum} has a nonzero root of multiplicity $m_2$, and the one on the left hand side has no nonzero root of multiplicity greater than $\ell$ by Lemma~\ref{Hajos}. Thus, $m_2\leq \ell$. Assume that $n_1=m_1\geq {\ell+2 \choose 2}+\ell-1$, so 
\[
m_1-m_2\geq {\ell+2 \choose 2}-1\geq \ell+2
\]
 since $\ell\geq 3$. Note that the coefficients of the polynomial on the right hand side in \eqref{linearsum} next to $x^{m_1}, x^{m_1-1},\ldots, x^{m_2+1}$ are all nonzero, since $\deg b_2h(x)^{m_2}=m_2$, and $h(0)\neq 0$. However, since $m_1-m_2\geq \ell+2$, this contradicts the assumption that on the left hand side in \eqref{linearsum} we have at most $\ell+1$ nonzero terms.
\end{proof}

For  $f\in K[x]$, by $\overline \Aut(f)$ we denote the group of linear polynomials $\mu(x)\in \overline{K}[x]$ for which $f\circ\mu=f$. So defined $\overline \Aut(f)$ is in  \cite[Section~6]{KZ14} called the automorphism group of $f$ over $\overline{K}$.
Let $f(x)=a_1x^{n_1}+\cdots+a_{\ell}x^{n_{\ell}}+a_{\ell+1}\in K[x]$ where $\ell\geq 1$ and
let further $d=\gcd(n_1, \ldots, n_{\ell})$. Clearly, $f(x)=(a_1x^{n_1/d}+\cdots+a_{\ell}x^{n_{\ell}/d}+a_{\ell+1})\circ  x^d$, and $f(x)=f(\mu(x))$ for any $\mu(x)=\zeta x$ where $\zeta^d=1$. Thus, $\abs{\overline \Aut(f)}\geq  d$. 
The following observations about the case when $\abs{\overline \Aut(f)}> d$ are of interest in relation to the case  $a_i=b_i$ for $i=1, 2, \ldots, \ell$, $a_{\ell+1}=0$ and $f(x):=a_1x^{n_1}+\cdots+a_{\ell}x^{n_{\ell}}+a_{\ell+1}$  in Theorem~\ref{main}. Namely,  if \eqref{triv} holds with $\mu(0)\neq 0$, then in this notation it must be that $\abs{\overline \Aut(f)}>d$. We will rely on results from \cite{ZM} and \cite{KZ14}.

\begin{proposition}\label{autf}
Let $K$ be a field with $\textnormal{char}(K)=0$, let $f\in K[x]$ have $\ell$ nonconstant terms and let $d$ denotes the greatest common divisor of the degrees of the terms of $f$. If $\abs{\overline \Aut(f)}> d$, then $\ell>d$. Furthermore, if $f$ is indecomposable, $\ell>1$ and $\abs{\overline \Aut(f)}> 1$, then $\deg f=\abs{\overline \Aut(f)}\leq \ell$.
\end{proposition}

\begin{proof}
We first show the first statement. Assume that $\abs{\overline \Aut(f)}> d$. Without loss of generality, we may assume that $K$ is algebraically closed. Then $\overline \Aut(f)$ is the automorphism group of $f$, as defined in \cite[Def.~6.1]{KZ14}.
By \cite[Lem.~6.3]{KZ14}, we may write $f(x)=g(h(x))$ where $K(x)/K(h(x))$ is Galois with Galois group $\overline \Aut(f)$ (see also the last paragraph of  \cite[Section~6]{KZ14}). Since $\textnormal{char}(K)=0$ and $K(x)/K(h(x))$ is Galois,  
by \cite[Lem.~3.3]{ZM}, \cite[Lem.~3.6]{ZM} and \cite[Rem.~5.2]{ZM}, it follows that $h$ is cyclic, that is $h(x)=\ell_1(x)\circ x^k\circ \ell_2(x)$ for some linear $\ell_1, \ell_2\in K[x]$ and $k\in \N$, and $\abs{\Mon(h)}=\deg h$. Then  $\abs{\overline \Aut(f)}=\abs{\Mon(h)}=\deg h$.  Assume first that $\ell_2(0)=0$. Then $h(x)=ax^k+b$ for some $a\in K\setminus\{0\}$ and $b\in K$, and thus $k\mid d$, so $\abs{\overline \Aut(f)}=k\leq d$, a contradiction with the assumption.
Assume henceforth $\ell_2(0)\neq 0$. Then since $f(x)=g(x)\circ \ell_1(x)\circ x^k\circ \ell_2(x)$, it follows that $f(x)-g(\ell_1(0))$ has a nonzero root of multiplicity $k$. By Lemma~\ref{Hajos} it follows that $k\leq \ell$. Thus, $d<\abs{\overline \Aut(f)}=k\leq \ell$. 

We now prove the second statement. If $f$ is indecomposable, then $f$ is also indecomposable over $\overline{K}$ (see \cite[p.~20]{S00}). In particular, $d=1$. By \cite[Cor.~6.6]{KZ14}, since $\abs{\overline \Aut(f)}> 1$ we have that  $f$ is cyclic and $\abs{\overline \Aut(f)}=\deg f$. Then $f(x)=\ell_1(x)\circ x^n\circ \ell_2(x)$ for some linear $\ell_1, \ell_2\in K[x]$ and $n:=\deg f$. If $\ell_2(0)=0$, then $\ell=1$, a contradiction. Thus, $\ell_2(0)\neq 0$ and $f(x)-\ell_1(0)$ has a nonzero root of multiplicity $n$. By Lemma~\ref{Hajos} it follows that $n\leq \ell$.
\end{proof}


In relation to the second statement of Proposition~\ref{autf}, note that if $\ell=1$, then $f$ is indecomposable if and only if $\deg f$ is a prime. In that case, $\abs{\overline \Aut(f)}=\deg f$.

\section{Diophantine equations and lacunary polynomials}\label{proofs}

In this section, we will prove our main results. From Theorem~\ref{T:BT} and Lemma~\ref{notDick} we first deduce the following proposition.

\begin{proposition}\label{doubly}
Let $K$ be a number field, $S$ a finite set of places of $K$ that contains all Archimedean places and $\mathcal{O}_S$  the ring of $S$-integers of $K$. If $f, g\in K[x]$ are such that $\deg f\geq 3, \deg g\geq 3$ and the absolute  monodromy groups of $f$ and $g$ are doubly transitive, the equation $f(x)=g(y)$ has infinitely many solutions with a bounded $\mathcal{O}_S$-denominator if and only if $f(x)=g(\mu(x))$ for some linear $\mu\in K[x]$.
\end{proposition}
\begin{proof}
If the equation $f(x)=g(y)$ has infinitely many solutions with a bounded $\mathcal{O}_S$-denominator, then by Theorem~\ref{T:BT} we have that
\begin{align}\label{condition0}
f(x)=\phi(f_1(\lambda(x))), \quad g(x)=\phi(g_1(\mu(x))),
\end{align}
for some $\phi, f_1, g_1,  \lambda, \mu\in K[x]$ such that $(f_1, g_1)$ is a standard or specific pair over $K$ and $\deg \lambda=\deg \mu=1$.

Assume that the absolute monodromy groups of $f$ and $g$ are doubly transitive. It follows, in particular, that $f$ and $g$ are indecomposable. 

Assume that $\deg \phi>1$. Then   from \eqref{condition0} it follows that $\deg f_1=1$ and $\deg g_1=1$, and $f(x)=g(\mu(x))$ for some linear $\mu\in K[x]$. If this holds, then the equation $f(x)=g(y)$ clearly has infinitely many solutions with a bounded $\mathcal{O}_S$-denominator, e.g.\@ set $x=\mu(t), y=t$, where $t\in \mathcal{O}_S$.

If $\deg \phi=1$, then from \eqref{condition0} it follows that
\begin{equation}\label{morse1}
f(x)=e_1f_1(c_1x+c_0)+e_0,\quad g(x)=e_1g_1(d_1x+d_0)+e_0,
\end{equation}
for some $c_1,c_0, d_1, d_0, e_1, e_0\in K$ such that $c_1d_1e_1\neq 0$.  Let $k:=\deg f=\deg f_1$ and $l:=\deg g=\deg g_1$. By assumption $k, l\geq 3$.

Note that $(f_1, g_1)$ is not a standard pair of the second kind since $k, l>2$. 

Furthermore, $(f_1, g_1)$ is not a standard pair of the fifth kind, since otherwise either  $f_1(x)=(ax^2-1)^3$ or $g_1(x)=(ax^2-1)^3$, so by \eqref{morse1} either $f$ or $g$ are decomposable, a contradiction with the assumption.

Also, $(f_1, g_1)$ is not a standard pair of the first kind, since by Lemma~\ref{notDick} and \eqref{morse1} neither $f_1(x)=x^k$ nor $g_1(x)=x^l$ is possible (since $k, l\geq 3$).

It also follows that $(f_1, g_1)$ is not a standard pair of the third or of the fourth kind. Namely, otherwise $\gcd(k, l)\leq 2$, and since $k, l\geq 3$, it follows that either $k\geq 4$ or $l\geq 4$, which together with \eqref{morse1} contradicts Lemma~\ref{notDick}.

In the same way, Lemma~\ref{notDick} implies that if $(f_1, g_1)$ is a specific pair, then $(k, l)=(3, 3)$. In this case, $\gcd(k, l)=3$, so $f_1(x)=D_3(x, a)=x^3-3xa$ and $g_1(x)=-D_3(1/2x, a)=-1/8x^3+3/2xa$, 
so $g_1(-2x)=f_1(x)$. Then from \eqref{morse1} it follows that $g(\mu(x))=f(x)$ for some linear $\mu\in K[x]$.
\end{proof}

We now  give a short proof of Theorem~\ref{tri2}.

\begin{proof}[Proof of Theorem~\ref{tri2}]
The only if part of Theorem~\ref{tri2} follows from Proposition~\ref{doubly} and Proposition~\ref{Montri}. 
Assume now that \eqref{lineartri2} holds. Then the equation clearly has infinitely many solutions $x, y\in K$ with a bounded $\mathcal{O}_S$-denominator.  Further assume without loss of generality that $n_2\leq m_2$. 
By Lemma~\ref{choose} it follows that if $\mu(0)\neq 0$, then $n_1=m_1\leq 3$. Thus, $n_1=m_1=3$. By comparing coefficients one easily works out that only the listed cases are possible.  If $\mu(0)=0$, then the last statement clearly holds.
\end{proof}

We now prove Theorem~\ref{main2}.

\begin{proof}[Proof of Theorem~\ref{main2}]
If the equation has infinitely many solutions with a bounded $\mathcal{O}_S$-denominator, then
\begin{align}\label{condition4}
a_{1}x^{n_{1}}+\cdots+a_{\ell}x^{n_{\ell}}+a_{\ell+1}&=\phi(f_1(\lambda(x))), \\
b_1x^{m_1}+b_2x^{m_2}&=\phi(g_1(\mu(x))),
\end{align}
for some $f_1, g_1, \phi, \lambda, \mu\in K[x]$ such that $(f_1, g_1)$ is a standard or specific pair over $K$ and $\deg \lambda=\deg \mu=1$. 

Assume that $\deg \phi>1$.  Since $\gcd(m_1, m_2)=1$,  by Proposition~\ref{Montri} it follows that $b_1x^{m_1}+b_2x^{m_2}$ is indecomposable. Thus, $\deg g_1=1$ and hence $\phi(x)=b_1\sigma(x)^{m_1}+b_2\sigma(x)^{m_2}$ for some linear $\sigma\in K[x]$. Then 
\[
a_{1}x^{n_{1}}+\cdots+a_{\ell}x^{n_{\ell}}+a_{\ell+1}=(b_1x^{m_1}+b_2x^{m_2}) \circ \sigma(f_1(\lambda(x))).
\]
By Proposition~\ref{nonlinear} it follows that  $m_1<{\ell+2 \choose 2}+\ell-1$, which contradicts the assumption.

Thus $\deg \phi=1$. Then
\begin{align}
\label{f21} a_{1}x^{n_{1}}+\cdots+a_{\ell}x^{n_{\ell}}+a_{\ell+1}&=e_1f_1(c_1x+c_0)+e_0,\\
\label{f22} b_1x^{m_1}+b_2x^{m_2}&=e_1g_1(d_1x+d_0)+e_0,
\end{align}
for some $c_1,c_0, d_1, d_0, e_1, e_0\in K$ such that $c_1d_1e_1\neq 0$. In particular, $\deg f_1=n_1$ and $\deg g_1=m_1$. By assumption, $m_1\geq 12$ and $n_1\geq 3$.  Note that by Proposition~\ref{Montri} and \eqref{f22}, the absolute monodromy group of $g_1$ is doubly transitive.

Now, $(f_1, g_1)$ is not a standard pair of the second kind since $n_1>2$ and $m_1>2$. 

Furthermore, $(f_1, g_1)$ is not a standard pair of the fifth kind since $m_1>6$.


Also, $(f_1, g_1)$ cannot be a standard pair of the third or of the fourth kind, nor a specific pair. Namely, recall that the absolute monodromy group of $g_1$ is doubly transitive, so the statement follows by  Lemma~\ref{notDick}, since $m_1\geq 12$.

If $(f_1,g_1)$ is a standard pair of the first kind, then either $g_1(x)=x^{m_1}$ or $f_1(x)=x^{n_1}$. Since the absolute monodromy group of $g_1$ is doubly transitive and $m_1\geq 12$, by Lemma~\ref{notDick} it follows that it must be that $f_1(x)=x^{n_1}$. Hence, 
\[
a_{1}x^{n_{1}}+\cdots+a_{\ell}x^{n_{\ell}}+a_{\ell+1}-e_0=e_1(c_1x+c_0)^{n_1},
\]
 so $c_0\neq 0$ and $n_1=\ell$.
Then $g_1(x)=c'x^rp(x)^{n_1}$ for some $c'\in K\setminus\{0\}$, $r<n_1$, $\gcd(r, n_1)=1$ and $r+\deg p>0$. Since the absolute monodromy group of $g_1$ is doubly transitive and $m_1\geq 12$, by Lemma~\ref{notDick} it follows that $\deg p>0$.
Then 
\[
b_1x^{m_1}+b_2x^{m_2}-e_0=e_1g_1(d_1x+d_0)=e_1c'(d_1x+d_0)^rp(d_1x+d_0)^{n_1}.
\]
Since $n_1\geq 3$,  by Lemma~\ref{Hajos} it follows that $p(d_1x+d_0)$ has no nonzero root. Then, since $n_1\geq 3$ and $\deg p>0$, it follows that $e_0=0$ and that $(d_1x+d_0)^r$ has exactly two terms, so $r=1$ and $d_0\neq 0$. Thus, \eqref{second} holds.

When \eqref{second} holds, there are infinitely many solutions $x, y$ with a bounded $\mathcal{O}_S$-denominator of the equation, since the equation $x^{n_1}=cy\mu(y)^{m_1-1}$ with linear $\mu\in K[x]$  has infinitely many solutions $x, y$ with a bounded $\mathcal{O}_S$-denominator. Namely, if $q, s\in \N$ are such that $qn_1=s+1$, then an infinite family of solutions is given by $x=c^qu\mu(c^su^{n_1})^{m_1-1}$, $y=c^su^{n_1}$, for $u\in \mathcal{O}_S$. 
\end{proof}

To the proof of Theorem~\ref{main} we need one more lemma. 

\begin{lemma}\label{Dicksoni}
Let $K$ be a number field. Assume that
\[
a_1x^{n_1}+\cdots+a_{\ell}x^{n_{\ell}}+a_{\ell+1}=e_1D_{n_1}(c_1x+c_0, \alpha)+e_0
\]
where  $\ell\geq 2$, $a_i, e_i, c_i, \alpha\in K$, and $e_1c_1\alpha\neq 0$.
Then $n_{i-1}-n_{i}\leq 2$ for all $i=2, 3, \ldots, \ell+1$, and thus $n_1\leq 2\ell$. 
\end{lemma}

\begin{proof}
By Lemma~\ref{choose}, for $i=2, \ldots, \ell+1$, the $(n_i+1)$-st derivative of $e_1D_{n_1}(c_1x+c_0, \alpha)+e_0$  has a nonzero root of multiplicity $n_{i-1}-n_{i}-1$. Thus, the $(n_i+1)$-st derivative of $D_{n_1}(c_1x+c_0, \alpha)$  has a nonzero root of multiplicity $n_{i-1}-n_{i}-1$.
We now show that $D_{n_1}^{(k)}(x, \alpha)$ has only simple roots for all $k=0, 1, \ldots, n_1-1$, so that $D_{n_1}^{(k)}(c_1x+c_0, \alpha)$ has only simple roots for all $k=0, 1, \ldots, n_1-1$.
Recall that $D_{n_1}(x,\alpha)=2\alpha^{n_1/2}T_{n_1}(x/(2\sqrt{\alpha}))$ where $T_k(x)=\cos (k \arccos x)$ is the $k$-th Chebyshev polynomial of the first kind. The roots of  $T_k(x)=\cos (k \arccos x)$ are $x_j:=\cos(\pi(2j-1)/(2k))$, $j=1,2,\dots,k$. These are all simple and real, so the roots of $T_{n_1}^{(k)}(x)$ are simple and real for all $k=0, 1, \ldots, n_1-1$, by Rolle's theorem. Since 
\[
D_{n_1}^{(k)}(x,\alpha)=\frac{2\alpha^{n_1/2}}{(2\sqrt{\alpha})^k}T_{n_1}^{(k)}(x/(2\sqrt{\alpha})),
\]
it follows that $D_{n_1}^{(k)}(x, \alpha)$ has only simple roots for all $k=0, 1, \ldots, n_1-1$. Note that the multiplicity of a nonzero root of $D_{n_1}^{(n_1)}(x, \alpha)$ is $0$. Therefore, $n_{i-1}-n_{i}-1\leq 1$ for all $i=2, \ldots, \ell+1$, and
\[
n_1=(n_1-n_2)+(n_2-n_3)+\cdots+(n_{\ell}-n_{\ell+1})\leq 2\ell.
\]
\end{proof}

The last statement of Lemma~\ref{Dicksoni} is shown in \cite{G16+}, for the case $K=\Q$, by using Lemma~\ref{Gawron}.

\begin{proof}[Proof of Theorem~\ref{main}]
If the equation has infinitely many solutions with a bounded $\mathcal{O}_S$-denominator, then
\begin{align}\label{condition4}
a_{1}x^{n_{1}}+\cdots+a_{\ell}x^{n_{\ell}}+a_{\ell+1}&=\phi(f_1(\lambda(x))), \\
b_1x^{m_1}+\cdots+b_{k}x^{m_k}&=\phi(g_1(\mu(x))),
\end{align}
for some $f_1, g_1, \phi, \lambda, \mu\in K[x]$ such that $(f_1, g_1)$ is a standard or specific pair over $K$ and $\deg \lambda=\deg \mu=1$. 

Assume that $\deg \phi>1$.  Since $b_1x^{m_1}+\cdots+b_{k}x^{m_k}$ is indecomposable it follows that $\deg g_1=1$, so that $\phi(x)=b_1\sigma(x)^{m_1}+\cdots+b_k\sigma(x)^{m_k}$ for some linear $\sigma\in K[x]$. Then 
\[
a_{1}x^{n_{1}}+\cdots+a_{\ell}x^{n_{\ell}}+a_{\ell+1}=(b_1x^{m_1}+\cdots+b_{k}x^{m_k}) \circ \sigma(f_1(\lambda(x))).
\]

From Theorem~\ref{Zannier} it follows that either $\sigma(f_1(\lambda(x)))=\zeta x^k+\nu$ for some $\zeta, \nu\in K$, or $m_1<2\ell(\ell-1)$. The latter cannot be by assumption. Note that if the former holds, then $k\mid n_i$ for all $i=1, 2, \ldots, {\ell}$. This contradicts the assumption on coprimality of $n_i$'s, unless $k=1$. If $k=1$, then \eqref{triv} holds, and the equation clearly has infinitely many solutions $x, y\in K$ with a bounded $\mathcal{O}_S$-denominator 

Assume henceforth $\deg \phi=1$. Then
\begin{align}
\label{f1} a_{1}x^{n_{1}}+\cdots+a_{\ell}x^{n_{\ell}}+a_{\ell+1}&=e_1f_1(c_1x+c_0)+e_0,\\
\label{f2} b_1x^{m_1}+\cdots+b_{k}x^{m_k}&=e_1g_1(d_1x+d_0)+e_0,
\end{align}
for some $c_1,c_0, d_1, d_0, e_1, e_0\in K$ such that $c_1d_1e_1\neq 0$. In particular, $\deg f_1=n_1$ and $\deg g_1=m_1$. By assumption $m_1\geq 12$ and $n_1\geq 3$. 

Note that $(f_1, g_1)$ is not a standard pair of the second kind, since $n_1>2$ and $m_1>2$. Similarly, $(f_1, g_1)$ is not a standard pair of the fifth kind since $m_1>6$. 

Also, $(f_1, g_1)$ is not a standard pair of the third or of the fourth kind, nor a specific pair. Namely, otherwise, by \eqref{f1} and  \eqref{f2}, and Lemma~\ref{Dicksoni}, it follows that  $n_1\leq 2\ell$ and $m_1\leq 2k$, a contradiction with the assumption.

Finally, if $(f_1, g_1)$ is  a standard pair of the first kind, then either  $f_1(x)=x^{n_1}$ or $g_1(x)=x^{m_1}$. Assume first that the former holds. Then
\begin{align}\label{thecase}
\begin{split}
a_{1}x^{n_{1}}+\cdots+a_{\ell}x^{n_{\ell}}+a_{\ell+1}-e_0&=e_1(c_1x+c_0)^{n_1},\\
b_1x^{m_1}+\cdots+b_{k}x^{m_k}-e_0&=e_1c'(d_1x+d_0)^rp(d_1x+d_0)^{n_1},
\end{split}
\end{align}
where $p\in K[x]$, $r<n_1$, $\gcd(r, n_1)=1$, $r+\deg p>0$ and $c'\neq 0$. Clearly, $n_1=\ell$ and $c_0\neq 0$. By Lemma~\ref{Hajos} it follows that either $p(d_1x+d_0)$ has no nonzero root, or $n_1\leq k$. If $n_1>k$, then we have
\[
b_1x^{m_1}+\cdots+b_{k}x^{m_k}-e_0=e_1c(d_1x+d_0)^rx^{m_1-r},
\]
for some $c\neq 0$. Then $r+1=k$.

Assume now that $g_1(x)=x^{m_1}$. Then 
\begin{align}
\begin{split}
a_{1}x^{n_{1}}+\cdots+a_{\ell}x^{n_{\ell}}+a_{\ell+1}-e_0&=e_1c'(c_1x+c_0)^rp(c_1x+c_0)^{m_1},\\
b_1x^{m_1}+\cdots+b_{k}x^{m_k}-e_0&=e_1(d_1x+d_0)^{m_1},
\end{split}
\end{align}
where $p\in K[x]$, $r<m_1$, $\gcd(r, m_1)=1$, $r+\deg p>0$ and $c'\neq 0$. Clearly, $m_1=k$ and $d_0\neq 0$. 
By Lemma~\ref{Hajos} it follows that either $p(c_1x+c_0)$ has no nonzero root, or $m_1\leq \ell$. The latter cannot be by assumption, so
\[
a_{1}x^{n_{1}}+\cdots+a_{\ell}x^{n_{\ell}}+a_{\ell+1}-e_0=e_1c(c_1x+c_0)^rx^{n_1-r}, 
\]
for some $c\neq 0$. Then $r+1=\ell$. 

Since $n_1\neq \ell$ and $m_1\neq k$ by assumption, we have that $(f_1, g_1)$ is  not a standard pair of the first kind. This completes the proof.
\end{proof}

We now discuss how the assumptions of Theorem~\ref{main} can be relaxed. 

Instead of requiring that either $m_1\geq 2k+1$ or $n_1\geq 2\ell+1$, we could have required that either there exists $i\in \{2, 3, \ldots, \ell+1\}$ such that $n_{i-1}-n_i>2$ or that there exists $i\in \{2, 3, \ldots, k+1\}$ such that $m_{i-1}-m_i>2$. This follows by Lemma~\ref{Dicksoni}, since we used this assumption only to eliminiate the cases when $\deg \phi=1$ and $(f_1, g_1)$ is a either standard pair of the third or fourth kind, or a specific pair. 

Instead of requiring that $n_1\neq \ell$ and $m_1\neq k$, we can list the cases that occur when $n_1=\ell$ or $m_1=k$, as was done in the last  paragraphs of the proof of Theorem~\ref{main}, and in Theorem~\ref{main2}.

If we assume that $m_1$ is a composite number, because of the assumption $ii)$, we can immediately eliminate the case  when $\deg \phi=1$ and $(f_1, g_1)$ is either a standard pair of the third or fourth kind, or a specific pair, since a Dickson polynomial of composite degree is decomposable (see e.g.~\cite[Lemma~1.1]{T95}). Thus we do not need to assume that either $m_1\geq 2k+1$ or $n_1\geq 2\ell+1$. In the same way, we do not need to assume that $m_1\neq k$, since this assumption serves to eliminate the case $b_1y^{m_1}+\cdots+b_{k}y^{m_k}-e_0=e_1(d_1x+d_0)^{m_1}$. (This cannot be since on the left hand side we have an indecomposable polynomial, and on the right a decomposable polynomial, since $m_1$ is by assumption composite). Thus, if we assume that $m_1$ is composite and relax the assumption $iii)$ to requiring that $m_1\geq 2\ell(\ell-1)$, we have that the equation \eqref{central2} has infinitely many solutions $x, y\in K$ with a bounded $\mathcal{O}_S$-denominator if and only if either \eqref{triv}  or \eqref{thecase} holds.

\section*{Appendix}\label{appendix}

In \cite{DG06, DGT05}, it is shown that a lacunary polynomial $a_1x^{n_1}+\cdots+a_{\ell}x^{n_{\ell}}+a_{\ell+1}$, where $a_i$'s are integers, is indecomposable (over $\Q$) when either $n_2=n_1-1$ and $\gcd(n_1, a_2)=1$,  or $f$ is an odd polynomial, $n_2=n_1-2$ and $\gcd(n_1, a_2)=1$. (By the convention set in the introduction, $n_i$'s are positive integers such that $n_i>n_j$ if $ i>j$ and $a_1a_2\cdots a_{\ell}\neq 0$.)
We now extend these results to the case when $a_i$'s are in a unique factorization domain of characteristic zero. This is of interest in relation to Theorem~\ref{main}.

Let $R$ be an integral domain and $L$ be its quotient field. Assume that $\charp(L)=0$. Let $K$ be any extension of $L$. 
For a nonconstant  $f\in R[x]$, write  $f(x)=g(h(x))$ with $g, h\in K[x]$, $\deg g\geq 2$, $\deg h\geq 2$, $h$ monic and $h(0)=0$. 
 Turnwald~\cite{T95} showed that then the coefficients of $g$ and $h$ belong to an integral closure of $R$ in $L$. If $R$ is a unique factorization domain, then $R$ is integrally closed in $L$, so the coefficients of $g$ and $h$ belong to $R$, and the following holds.

\begin{corollary}\label{UFD}
Let $R$ be a unique factorization domain and $K$ any field extension of the quotient field of $R$. Assume that $\charp(K)=0$.  Let $f\in R[x]$ be such that 
 $f(x)=g(h(x))$ with $g, h\in K[x]$, $\deg g\geq 2$, $\deg h\geq 2$, $h$ monic and $h(0)=0$. Then $g, h\in R[x]$.
\end{corollary}

In particular, if $K$ is a number field of class number $1$, $R$ is the ring of algebraic integers of $K$ and $f\in R[x]$ is such that $f(x)=g(h(x))$, where $g, h\in K[x]$, $\deg g\geq 2, \deg h\geq 2$, $h$ monic and $h(0)=0$, then $g, h\in R[x]$.

 Turnwald \cite{T95} further showed that if number field $K$ is of class number greater than $1$ and $R$ is the ring of algebraic integers of $K$, then for every prime $q$ there exists $f\in R[x]$ of degree $q^2$ which is decomposable over $K$, but cannot be represented as a composition of polynomials  in $R[x]$. 

We now prove the sought result. In the sequel, for a unique factorization domain $R$, $t\in \Z$ and $a\in R$, we say that $t$ divides $a$ in $R$, and write $t\mid a$ in $R$, when there exists $a'\in R$ such that $a=ta'$. 
\begin{proposition}\label{GCD}
Let $R$ be a unique factorization domain and $K$ any field extension of the quotient field of $R$. Assume that  $\charp(K)=0$. Let $f(x)=a_1x^{n_1}+\cdots+a_{\ell}x^{n_{\ell}}+a_{\ell+1}\in R[x]$, where  $n_i$'s are distinct  positive integers with $n_i>n_j$ for $i>j$, and  $a_1a_2\cdots a_{\ell}\neq 0$. Assume that
 $f(x)=g(h(x))$, where $g, h\in K[x]$, $\deg g\geq 2$ and $\deg h\geq 2$.
Then either $h(x)=\zeta x^{m}+\nu$ for some $\zeta, \nu\in K$ and $m\mid n_i$ for all $i=1,2, \ldots, \ell$, or $\deg g\mid a_2$ in $R$.

 In particular, if $\gcd(n_1, \ldots, n_{\ell})=1$  and there does not exist integer $t\geq 2$ such that $t\mid n_1$ and $t\mid a_2$ in $R$, then $f$ is indecomposable over $K$.
\end{proposition}
\begin{proof}
Let $f(x)=g(h(x))$ with $g, h\in K[x]$, $\deg g\geq 2$, $\deg h\geq 2$, $h$ monic and $h(0)=0$.
By Corollary~\ref{UFD} it follows that $g, h\in R[x]$. Let $\deg h=m$ and $\deg g=t$. Let further $h(x)=b_1x^{m_1}+b_2x^{m_2}+\cdots+b_{k}x^{m_k}$ with $m_i\in \N$ and $b_i\in R\setminus \{0\}$. By assumption, $b_1=1$ and $m_1=m$.
If $h(x)=x^{m}$, then clearly $m\mid n_i$ for all $i=1,2, \ldots, \ell$.
Assume that $h$ is not a monomial, so that $k\geq 2$. Then by $f(x)=g(h(x))$, it follows that $f(x)=a_1h(x)^t+p(x)$, where $p\in R[x]$, $\deg p\leq (t-1)m_1$. Then $a_2=a_1tb_2$ by comparison of coefficients on both sides next to $x^{n_2}$. Thus $t\mid a_2$ in $R$. Clearly, $t\mid n_1$ as well.
\end{proof}

\section*{Acknowledgements}
I am thankful for the support of the Austrian Science Fund (FWF) through projects W1230, F-5510 and FWF-24574.

\end{document}